\documentclass[12pt]{amsart}
\usepackage{fullpage,url,amssymb,enumerate,colonequals}
 \usepackage[all]{xy} 
\usepackage{mathrsfs} 

\usepackage[OT2,T1]{fontenc}
\DeclareSymbolFont{cyrletters}{OT2}{wncyr}{m}{n}
\DeclareMathSymbol{\Sha}{\mathalpha}{cyrletters}{"58}

\usepackage{color}

\newcommand{\defi}[1]{\textsf{#1}} 

\newcommand{\Aff}{\mathbb{A}}
\newcommand{\C}{\mathbb{C}}
\newcommand{\F}{\mathbb{F}}
\newcommand{\fhat}{\widehat{f}}

\newcommand{\PP}{\mathbb{P}}
\newcommand{\Q}{\mathbb{Q}}
\newcommand{\R}{\mathbb{R}}

\newcommand{\Z}{\mathbb{Z}}

\newcommand{\calE}{\mathcal{E}}

\newcommand{\calW}{\mathcal{W}}

\newcommand{\DD}{\mathscr{D}}

\newcommand{\Pointed}{\mathscr{P}}
\newcommand{\UU}{\mathscr{U}}

\newcommand{\an}{{\operatorname{an}}}
\newcommand{\anm}{\textup{an}-}

\newcommand{\injects}{\hookrightarrow}
\newcommand{\intersect}{\cap} 

\newcommand{\union}{\cup} 
\newcommand{\Union}{\bigcup} 

\newtheorem{theorem}{Theorem}[section]
\newtheorem{lemma}[theorem]{Lemma}
\newtheorem{corollary}[theorem]{Corollary}
\newtheorem{proposition}[theorem]{Proposition}

\theoremstyle{definition}

\theoremstyle{remark}
\newtheorem{remark}[theorem]{Remark}

\newtheorem{warning}[theorem]{Warning}

\usepackage[
	backref,
	pdfauthor={Ehud Hrushovski, Francois Loeser, Bjorn Poonen}, 
]{hyperref}
\usepackage[alphabetic,backrefs,lite]{amsrefs} 

\begin{document}

\title{Berkovich spaces embed in Euclidean spaces}

\author{Ehud Hrushovski}
\thanks{E.H. was supported by 
the European Research Council under the European Union's Seventh Framework Programme (FP7/2007-2013) / ERC Grant agreement no. 291111/ MODAG} 
 \address{Department of Mathematics, The Hebrew University, Jerusalem, Israel} \email{ehud@math.huji.ac.il}
\author{Fran\c{c}ois Loeser}
\thanks{F.L. was supported by the European Research Council under the European Union's Seventh Framework Programme (FP7/2007-2013) / ERC Grant agreement no. 246903/NMNAG}
\address{Institut de Math\'ematiques de Jussieu, 
UMR 7586 du CNRS,
Universit\'e Pierre et Marie Curie, 
Paris, France}
\email{Francois.Loeser@upmc.fr} 
\urladdr{http://www.dma.ens.fr/~loeser/}
\author{Bjorn Poonen}
\thanks{B.P. was partially supported by the Guggenheim Foundation and National Science Foundation grant DMS-1069236.}
\address{Department of Mathematics, Massachusetts Institute of Technology, Cambridge, MA 02139-4307, USA}
\email{poonen@math.mit.edu}
\urladdr{\url{http://math.mit.edu/~poonen/}}

\date{October 23, 2012}

\begin{abstract}
Let $K$ be a field that is complete with respect to 
a nonarchimedean absolute value
such that $K$ has a countable dense subset.
We prove that the Berkovich analytification $V^{\an}$ of any 
$d$-dimensional quasi-projective scheme $V$ over $K$ embeds in $\R^{2d+1}$.
If, moreover, the value group of $K$ is dense in $\R_{>0}$
and $V$ is a curve,
then we describe the homeomorphism type of $V^{\an}$
by using the theory of local dendrites.
\end{abstract}

\maketitle

\section{Introduction}\label{S:introduction}

In this article, \defi{valued field} will mean
a field $K$ equipped with a nonarchimedean absolute value $|\;|$
(or equivalently with a valuation 
taking values in an additive subgroup of $\R$).
Let $K$ be a complete valued field.
Let $V$ be a quasi-projective $K$-scheme.
The associated Berkovich space $V^{\an}$ \cite{Berkovich1990}*{\S3.4} 
is a topological space 
that serves as a nonarchimedean analogue of the complex analytic space
associated to a complex variety.
(Actually, $V^{\an}$ carries more structure, but it is only the
underlying topological space that concerns us here.)
Although the set $V(K)$ in its natural topology is totally disconnected,
$V^{\an}$ is arcwise connected if and only if 
$V$ is connected~\cite{Berkovich1990}*{Proposition~3.4.8(iii)}.
Also, $V^{\an}$ is locally contractible: 
see \cites{Berkovich1999,Berkovich2004} for the smooth case,
and \cite{Hrushovski-Loeser-preprint}*{Theorem~13.4.1} for the general case.

Our goal is to study the topology of $V^{\an}$
under a mild countability hypothesis on $K$ with its absolute value topology.
For instance, we prove the following:

\begin{theorem}
\label{T:main}
Let $K$ be a complete valued field having a countable dense subset.
Let $V$ be a quasi-projective $K$-scheme of dimension $d$.
Then $V^{\an}$ is homeomorphic to a subspace of $\R^{2d+1}$.
\end{theorem}

\begin{remark}
The hypothesis that $K$ has a countable dense subset 
is necessary as well as sufficient.
Namely, $K$ embeds in $(\Aff^1_K)^{\an}$, 
so if the latter embeds in a separable metric space such as $\R^n$,
then $K$ must have a countable dense subset.
\end{remark}

\begin{remark}
The hypothesis is satisfied when $K$ is the completion
of an algebraic closure of a completion of a global field $k$, 
i.e., when $K$ is $\C_p \colonequals \widehat{\overline{\Q}}_p$ 
or its characteristic~$p$ analogue $\widehat{\overline{\F_p((t))}}$,
because the algebraic closure of $k$ in $K$ is countable and dense.
It follows that the hypothesis is satisfied also 
for any complete subfield of these two fields.
\end{remark}

Recall that a valued field is called \defi{spherically complete} 
if every descending sequence of balls has nonempty intersection.
Say that $K$ has \defi{dense value group}
if $|\;| \colon K^\times \to \R_{>0}$ has dense image,
or equivalently if the value group is not isomorphic to $\{0\}$ or $\Z$.

\begin{remark}
The separability 
hypothesis fails 
for any spherically complete field $K$ with dense value group.
Proof: 
Let $(t_i)$ be a sequence of elements of $K$ 
such that the sequence $|t_i|$ is strictly decreasing with positive limit.
For each sequence $\epsilon = (\epsilon_i)$ 
with $\epsilon_i \in \{0,1\}$,
define 
\[
	U_\epsilon \colonequals \left\{ x \in K: \left| x- \textstyle\sum_{i=1}^n \epsilon_i t_i \right| < |t_n| \textup{ for all $n$} \right\}.
\]
The $U_\epsilon$ are uncountably many disjoint open subsets of $K$,
and each is nonempty by definition of spherically complete.
\end{remark}

Let us sketch the proof of Theorem~\ref{T:main}.
We may assume that $V$ is projective.
The key is a result that presents $V^{\an}$ as a filtered limit of 
finite simplicial complexes.
Variants of this limit description 
have appeared in several places in the literature
(see the end of \cite{Payne2009}*{Section~1} for a summary);
for convenience, we use \cite{Hrushovski-Loeser-preprint}*{Theorem~13.2.4},
a version that does not assume that $K$ is algebraically closed
(and that proves more than we need, namely that the maps in the inverse
limit can be taken to be strong deformation retractions).
Our hypothesis on $K$ is used to show that the index set for the limit 
has a countable cofinal subset.
To complete the proof, we use a well-known result from topology,
Proposition~\ref{P:embedding inverse limit},
that an inverse limit of a sequence of finite simplicial complexes
of dimension at most $d$ can be embedded in $\R^{2d+1}$.

Our article is organized as follows.
Sections \ref{S:embeddings} and~\ref{S:limits}
give a quick proof of Proposition~\ref{P:embedding inverse limit}.
Sections \ref{S:fiber coproducts} and~\ref{S:noncomplete fields}
prove results that are needed to replace $K$ by a countable subfield,
in order to obtain a countable index set for the inverse limit.
Section~\ref{S:embeddings of Berkovich spaces} combines all of the above
to prove Theorem~\ref{T:main}.
The final sections of the paper study the topology of Berkovich curves:
after reviewing and developing the theory of 
dendrites and local dendrites 
in Sections \ref{S:dendrites} and~\ref{S:local dendrites},
respectively,
we show in Section~\ref{S:Berkovich curves}
how to obtain the homeomorphism type of any Berkovich curve
over $K$ as above, under the additional hypothesis that the 
value group is dense in $\R_{>0}$.
For example, as a special case of Corollary~\ref{C:P^1},
we show that $(\PP^1_{\C_p})^{\an}$ is homeomorphic
to a topological space first constructed in 1923, 
the Wa{\.z}ewski universal dendrite~\cite{Wazewski1923}

\section{Approximating maps of finite simplicial complexes by embeddings}
\label{S:embeddings}

If $X$ is a topological space,
a map $f \colon X \to \R^n$ is called an \defi{embedding} 
if $f$ is a homeomorphism onto its image.
For compact $X$, it is equivalent to require that $f$ be 
a continuous injection.
When we speak of a finite simplicial complex, we always mean its
geometric realization, a compact subset of some $\R^n$.
A set of points in $\R^n$ is said to be in \defi{general position}
if for each $m \le n-1$, no $m+2$ of the points lie in an 
$m$-dimensional affine subspace.

\begin{lemma}
\label{L:simplicial approximation}
Let $X$ be a finite simplicial complex of dimension at most $d$.
Let $\epsilon \in \R_{>0}$.
For any continuous map $f \colon X \to \R^{2d+1}$,
there is an embedding $g \colon X \to \R^{2d+1}$
such that $|g(x)-f(x)| \le \epsilon$ for all $x \in X$.
\end{lemma}

\begin{proof}
The simplicial approximation theorem implies that $f$
can be approximated within $\epsilon/2$ by a piecewise linear map $g_0$.
For each vertex $x_i$ in the corresponding subdivision of $X$, in turn,
choose $y_i \in \R^{2d+1}$ within $\epsilon/2$ of $g_0(x_i)$
so that the $y_i$ are in general position.
Let $g \colon X \to \R^{2d+1}$ be the piecewise linear map, 
for the same subdivision, such that $g(x_i)=y_i$.
Then $g$ is injective, and $g$ is within $\epsilon/2$ of $g_0$,
so $g$ is within $\epsilon$ of $f$.
\end{proof}

\section{Inverse limits of finite simplicial complexes}
\label{S:limits}

\begin{proposition}
\label{P:embedding inverse limit}
Let $(X_n)_{n \ge 0}$ be an inverse system of finite simplicial complexes
of dimension at most $d$
with respect to continuous maps $p_n \colon X_{n+1} \to X_n$.
Then the inverse limit $X \colonequals \varprojlim X_n$ embeds in $\R^{2d+1}$.
\end{proposition}

\begin{proof}
For $m \ge 0$, let $\Delta_m \subseteq X_m \times X_m$ be the diagonal,
and write $(X_m \times X_m) - \Delta_m = \Union_{n=m}^\infty C_{mn}$
with $C_{mn}$ compact.
For $0 \le m \le n$,
let $D_{mn}$ be the inverse image of $C_{mn}$ in $X_n \times X_n$.
Let $K_n = \Union_{m=1}^n D_{mn}$.
Since $K_n$ is closed in $X_n \times X_n$, it is compact.

For $n \ge 0$, 
we inductively construct an embedding $f_n \colon X_n \to \R^{2d+1}$
and numbers $\alpha_n,\epsilon_n \in \R_{>0}$
such that the following hold for all $n \ge 0$:
\begin{enumerate}[\upshape (i)]
\item\label{I:K_n} 
  If $(x,x') \in K_n$, then $|f_n(x)-f_n(x')| \ge \alpha_n$.
\item\label{I:delta alpha} 
  $\epsilon_n < \alpha_n/4$.
\item\label{I:deltas} 
  $\epsilon_n < \epsilon_{n-1}/2$ (if $n \ge 1$).
\item\label{I:n to n+1}
  If $x \in X_{n+1}$, 
  then $|f_{n+1}(x)-f_n(p_n(x))| \le \epsilon_n$.
\end{enumerate}
Let $f_0 \colon X_0 \to \R^{2d+1}$ be any embedding
(apply Lemma~\ref{L:simplicial approximation} to a constant map,
for instance).
Now suppose that $n \ge 0$ and that $f_n$ has been constructed.
Since $f_n$ is injective and $K_n$ is compact, 
we may choose $\alpha_n \in \R_{>0}$ satisfying~\eqref{I:K_n}.
Choose any $\epsilon_n \in \R_{>0}$ satisfying 
\eqref{I:delta alpha} and~\eqref{I:deltas}.
Apply Lemma~\ref{L:simplicial approximation} to $p_n \circ f_n$
to find $f_{n+1}$ satisfying~\eqref{I:n to n+1}.
This completes the inductive construction.

Now $\sum_{i=n}^\infty \epsilon_i < 2 \epsilon_n < \alpha_n/2$
by \eqref{I:deltas} and~\eqref{I:delta alpha}.
Let $\fhat_n$ be the composition $X \to X_n \stackrel{f_n}\to \R^{2d+1}$.
For $x \in X$, \eqref{I:n to n+1} implies 
$|\fhat_{n+1}(x) - \fhat_n(x)| \le \epsilon_n$,
so the maps $\fhat_n$ converge uniformly to a continuous 
map $f \colon X \to \R^{2d+1}$
satisfying $|f(x) - f_n(x_n)| < \alpha_n/2$.

We claim that $f$ is injective.
Suppose that $x = (x_n)$ and $x'=(x'_n)$ are distinct points of $X$.
Fix $m$ such that $x_m \ne x_m'$.
Fix $n \ge m$ such that $(x_m,x_m') \in C_{mn}$.
Then $(x_n,x'_n) \in D_{mn} \subseteq K_n$.
By~\eqref{I:K_n}, $|f_n(x_n)-f_n(x_n')| \ge \alpha_n$.
On the other hand, 
$|f(x)-f_n(x_n)| < \alpha_n/2$ and $|f(x')-f_n(x'_n)| < \alpha_n/2$,
so $f(x) \ne f(x')$.
\end{proof}

\begin{remark}
\label{R:history}
Proposition~\ref{P:embedding inverse limit} was proved in the 1930s.
Namely, following a 1928 sketch by K.~Menger,
in 1931 it was proved independently in 
by S.~Lefschetz~\cite{Lefschetz1931},
G.~N\"obeling~\cite{Noebeling1931},
and L.~Pontryagin and G.~Tolstowa~\cite{Pontrjagin-Tolstowa1931}
that any compact metrizable space of dimension at most $d$
embeds in $\R^{2d+1}$.
The proofs proceed by using P.~Alexandroff's idea of approximating
compact spaces by finite simplicial complexes (nerves of finite covers),
so even if it not obvious that the 1931 \emph{result} applies directly 
to an inverse limit of finite simplicial complexes of dimension
at most $d$
(i.e., whether such an inverse limit is of dimension at most $d$),
the \emph{proofs} still apply.
And in any case, in 1937 H.~Freudenthal~\cite{Freudenthal1937} proved
that a compact metrizable space is of dimension at most $d$
if and only if it is an inverse limit of finite simplicial complexes
of dimension at most $d$.
See Sections 1.11 and~1.13 of~\cite{Engelking1978}
for more about the history, including later improvements.
\end{remark}

\section{Fiber coproducts of valued fields}
\label{S:fiber coproducts}

We work in the category whose objects are valued fields
and whose morphisms are field homomorphisms respecting the absolute values.
For example, if $K$ is a valued field,
we have a natural morphism from $K$ to its completion $\widehat{K}$.
Given morphisms $i_1 \colon K \to L_1$ 
and $i_2 \colon K \to L_2$ of valued fields,
an \defi{amalgam} of $L_1$ and $L_2$ over $K$
is a triple $(M,j_1,j_2)$ where $M$ is a valued field
and $j_1 \colon L_1 \to M$ and $j_2 \colon L_2 \to M$ are morphisms
such that $j_1 \circ i_1 = j_2 \circ i_2$
and such that $M$ is generated by $j_1(L_1)$ and $j_2(L_2)$.
An \defi{isomorphism of amalgams} $(M,j_1,j_2) \to (M',j_1',j_2')$
is an isomorphism $\phi \colon M \to M'$ such that
$\phi \circ j_1 = j_1'$ and $\phi \circ j_2 = j_2'$.

\begin{proposition}
\label{P:fiber coproduct}
Given morphisms $K \to L_1$ and $K \to L_2$ of valued fields
such that $K$ is dense in $L_1$,
the fiber coproduct of $L_1$ and $L_2$ over $K$
exists and is the unique amalgam of $L_1$ and $L_2$ over $K$.
\end{proposition}

\begin{proof}
Since $K$ is dense in $L_1$,
the composition $K \to L_2 \to \widehat{L}_2$ 
extends uniquely to $L_1 \to \widehat{L}_2$.
Hence we may view $K$, $L_1$, and $L_2$
as subfields of $\widehat{L}_2$, and the morphisms
$K \to L_1 \to \widehat{L}_2$ 
and $K \to L_2 \to \widehat{L}_2$
as inclusions.
Let $M \colonequals L_1 L_2 \subseteq \widehat{L}_2$.
Now, given any $M'$ in a commutative diagram
\[
\xymatrix{
& M' \\
L_1 \ar[ru] && L_2 \ar[lu] \\
& K \ar[lu] \ar[ru] \\
}
\]
we may view one of the two upper morphisms, say $L_2 \to M'$, 
as an inclusion.
Then all the fields become subfields of $\widehat{M}'$.
Now the other upper morphism $L_1 \to M'$ is an inclusion too
since the composition
$L_1 \to M' \injects \widehat{M}'$
restricts to the inclusion morphism on the dense subfield $K$.
Thus $M = L_1 L_2 \subseteq M'$,
and there is a unique morphism $M \to M'$
compatible with the morphisms from $L_1$ and $L_2$,
namely the inclusion.
Thus $M$ is a fiber coproduct.

The existence of a fiber coproduct implies
that at most one amalgam exists.
Since $M$ is generated by $L_1$ and $L_2$, it is an amalgam.
\end{proof}

\begin{remark}
\label{R:amalgams in larger category}
In \cite{Hrushovski-Loeser-preprint},
the value groups of valued fields
are not necessarily contained in $\R$.
Proposition~\ref{P:fiber coproduct} remains true in the larger category,
and the amalgams in the two categories coincide
when they make sense,
i.e., when the valued fields in question happen
to have value group contained in $\R$.
\end{remark}

\section{Berkovich spaces over noncomplete fields}
\label{S:noncomplete fields}

Berkovich analytifications were originally defined only when 
the valued field $K$ was complete~\cite{Berkovich1990}.
For a quasi-projective variety $V$ 
over an \emph{arbitrary} valued field $K$,
\cite{Hrushovski-Loeser-preprint}*{Section~13.1}
defines a topological space in terms of types,
and proves that it is homeomorphic to $V^{\an}$ when $K$ is complete.
 
This definition uses types over $K \union \R$.  
Using quantifier elimination for the theory of 
algebraically closed valued fields in the two-sorted language 
consisting of the valued field and the value group,
such types can be identified with pairs $(L,c)$ 
with $L$ an $\R$-valued field extension of $K$,
where $(L,c)$ is identified with $(L',c')$ 
if there exists a $K$-isomorphism $f \colon L \to L'$ of $\R$-valued fields 
with $f(c)=c'$.   
This description makes it clear that if $v$ is the type of $(K',a)$, 
then an extension of $v$ to $L \geq K$ 
corresponds precisely to an amalgam of $K'$ and $L$ over $K$.  

The restriction map $r$ from types over $L$ to types over $K'$ 
(where $K \leq K' \leq L$)  takes $(L,a)$ to $(K'(a),a)$.
If $h \colon V \to W$ is a morphism of varieties over $K$, 
the restriction map $r$ is clearly compatible 
with the natural map from types on $V$ to types on $W$ induced by $h$. 

We take this space of types as a definition of 
the topological space $V^{\an}$ for arbitrary valued fields $K$.
The following proposition shows that no new spaces arise: 
it would have been equivalent
to define $V^{\an}$ as $(V_{\widehat{K}})^{\an}$ 
(the subscript denotes base extension).

\begin{proposition}
\label{P:base extension of Berkovich space}
Let $K \le L$ be an extension of valued fields such that $K$ is dense in $L$.
Let $V$ be a quasi-projective $K$-variety.
Then $(V_L)^{\an}$ is naturally homeomorphic to $V^{\an}$.
\end{proposition}

\begin{proof}
Restriction of types defines a continuous map
$r_V \colon (V_L)^{\an} \to V^{\an}$.
A point $v \in V^{\an}$ is represented by the type of some $a \in V(K')$
for some valued field extension $K'=K(a)$ of $K$;
then $r_V^{-1}(v)$ is in bijection with the set of 
amalgams of $K'$ and $L$ over $K$, 
which by Proposition~\ref{P:fiber coproduct} is a set of size~$1$.
Thus $r_V$ is a bijection.

If $V$ is projective, then $V^{\an}$ and $(V_L)^{\an}$ are
compact Hausdorff spaces~\cite{Hrushovski-Loeser-preprint}*{Proposition~13.1.2},
so the continuous bijection $r_V$ is a homeomorphism.
If $V$ is an open subscheme of a projective variety $P$,
then $(V_L)^{\an}$ and $V^{\an}$ are open subspaces 
of $(P_L)^{\an}$ and $P^{\an}$, respectively,
and $r_P$ restricts to $r_V$,
so the result for $P$ implies the result for $V$.
\end{proof}

\section{Embeddings of Berkovich spaces}
\label{S:embeddings of Berkovich spaces}

\begin{proposition}
\label{P:structure of V^an}
Let $K$ be a valued field having a countable dense subset.
Let $V$ be a \emph{projective} $K$-scheme of dimension $d$.
Then $V^{\an}$ is homeomorphic to 
a inverse limit $\varprojlim X_n$
where each $X_n$ is a finite simplicial
complex of dimension at most $d$
and each map $X_{n+1} \to X_n$ is continuous.
\end{proposition}

\begin{proof}
First suppose that $K$ is countable.
Since $V$ is projective, $V^{\an}$ is compact,
so we may apply \cite{Hrushovski-Loeser-preprint}*{Theorem~13.2.4} 
to $V^{\an}$ to obtain that $V^{\an}$ is 
a filtered limit of finite simplicial complexes over an index set $I$.
Since $K$ is countable, 
the proof of \cite{Hrushovski-Loeser-preprint}*{Theorem~13.2.4} 
shows that $I$ may be taken to be countable,
so our limit may be taken over a sequence,
as desired.

Now assume only that $K$ has a countable dense subset.
Since $V$ is of finite presentation over $K$,
it is the base extension of a projective scheme $V_0$
over a countable subfield $K_0$ of $K$.
By adjoining to $K_0$ a countable dense subset of $K$,
we may assume that $K_0$ is dense in $K$.
By Proposition~\ref{P:base extension of Berkovich space},
$V^{\an}$ is homeomorphic to $(V_0)^{\an}$,
which has already been shown to be an inverse limit of the desired form.
\end{proof}

\begin{proposition}
\label{P:open subscheme}
Let $K$ be a complete valued field.
If $U$ is an open subscheme of $V$,
then the induced map $U^{\an} \to V^{\an}$ is a homeomorphism
onto an open subspace.
\end{proposition}

\begin{proof}
See~\cite{Berkovich1990}*{Proposition~3.4.6(8)}.
\end{proof}

Theorem~\ref{T:main} follows immediately from 
Propositions \ref{P:embedding inverse limit}, 
\ref{P:structure of V^an}, and~\ref{P:open subscheme}.

\section{Dendrites}\label{S:dendrites}

When $V$ is a curve, more can be said about $V^{\an}$.
But first we recall some definitions and facts from topology.

\subsection{Definitions}

A \defi{continuum} is a compact connected metrizable space
(the empty space is not connected).
A \defi{simple closed curve} in a topological space
is any subspace homeomorphic to a circle.
A \defi{dendrite} is a locally connected continuum
containing no simple closed curve.
Dendrites may be thought of as topological generalizations of trees 
in which branching may occur at a dense set of points.
A point $x$ in a dendrite $X$ is called a \defi{branch point}
if $X-\{x\}$ has three or more connected components.

\subsection{Wa{\.z}ewski's theorems}

The following three theorems 
were proved by T.~Wa{\.z}ewski 
in his thesis~\cite{Wazewski1923}.\footnote{Actually, Wa{\.z}ewski 
used a different, equivalent definition: 
for him, a dendrite was any image $D$ of a continuous map 
$[0,1] \to \R^n$ such that $D$ contains no simple closed curve.
A dendrite in Wa{\.z}ewski's sense is a dendrite in our sense
by \cite{Nadler1992}*{Corollary~8.17}.
Conversely, a dendrite in our sense embeds in $\R^2$
by \cite{Nadler1992}*{Section~10.37}
(or, alternatively, is an inverse limit of finite trees
by \cite{Nadler1992}*{Theorem~10.27} 
and hence embeds in $\R^3$ by Proposition~\ref{P:embedding inverse limit}),
and is a continuous image of $[0,1]$ by the 
Hahn--Mazurkiewicz theorem \cite{Nadler1992}*{Theorem~8.14}.}

\begin{theorem}
\label{T:Wazewski}
Up to homeomorphism, there is a unique dendrite $W$
such that its branch points are dense in $W$
and there are $\aleph_0$ branches at each branch point.
\end{theorem}

The dendrite $W$ in Theorem~\ref{T:Wazewski}
is called the \defi{Wa{\.z}ewski universal dendrite}.

\begin{theorem}
\label{T:universal}
Every dendrite embeds in $W$.
\end{theorem}

\begin{theorem}
\label{T:traceable}
Every dendrite is homeomorphic to the image of some continuous map
$[0,1] \to \R^2$.
\end{theorem}

\begin{remark}
The key to drawing $W$ in the plane is to make sure that the
branches coming out of each branch point have diameters tending to $0$.
\end{remark}

\subsection{Pointed dendrites}

A \defi{pointed dendrite} is a pair $(X,P)$
where $X$ is a dendrite and $P \in X$.
An \defi{embedding of pointed dendrites} 
is an embedding of topological spaces mapping the point in the first
to the point in the second.
Let $\Pointed$ be the category of pointed dendrites,
in which morphisms are embeddings.
By the \defi{universal pointed dendrite},
we mean $W$ equipped with one of its branch points $w$.

\begin{theorem}
\label{T:universal pointed dendrite}
Every pointed dendrite $(X,P)$ 
admits an embedding into the universal pointed dendrite $(W,w)$.
\end{theorem}

\begin{proof}
Enlarge $X$ by attaching a segment at $P$ in order to assume that 
$P$ is a branch point of $X$.
Theorem~\ref{T:universal} yields an embedding $i \colon X \injects W$.
Then $i(P)$ is a branch point of $W$.
By \cite{Charatonik1991}*{Proposition~4.7},
there is a homeomorphism $j \colon W \to W$ mapping $i(P)$ to $w$.
Then $j \circ i$ is an embedding $(X,P) \to (W,w)$.
\end{proof}

\begin{proposition}
\label{P:contractibility of dendrites}
Any dendrite admits a strong deformation retraction 
onto any of its points.
\end{proposition}

\begin{proof}
In fact, a dendrite admits a strong deformation retraction
onto any subcontinuum~\cite{Illanes1996}.
\end{proof}

\section{Local dendrites}\label{S:local dendrites}

\subsection{Definition and basic properties}

A \defi{local dendrite} is a continuum such that every point has 
a neighborhood that is a dendrite.
Equivalently, a continuum is a local dendrite 
if and only if it is locally connected
and contains at most a finite number of 
simple closed curves~\cite{Kuratowski1968}*{\S51, VII, Theorem~4(i)}.
Local dendrites are generalizations of finite connected graphs,
just as dendrites are generalizations of finite trees.

\begin{proposition}
\label{P:properties of local dendrites}
\hfill
\begin{enumerate}[\upshape (a)]
\item \label{I:subcontinuum of local dendrite}
Every subcontinuum of a local dendrite is a local dendrite.
\item \label{I:arcwise connected local dendrite}
An open subset of a local dendrite is arcwise connected if and only if
it is connected.
\item \label{I:simply connected local dendrite}
A connected open subset $U$ of a local dendrite 
is simply connected if and only if
it contains no simple closed curve.
\item \label{I:dendrite vs local dendrite}
A dendrite is the same thing as a simply connected local dendrite.
\end{enumerate}
\end{proposition}

\begin{proof}\hfill
  \begin{enumerate}[\upshape (a)]
  \item
This follows from the fact that every subcontinuum of a dendrite is a 
dendrite~\cite{Kuratowski1968}*{\S51, VI, Theorem~4}.
  \item
This follows from \cite{Whyburn1971}*{II, (5.3)}.
  \item
If $U$ contains a simple closed curve $\gamma$, 
\cite{Borsuk-Jaworowski1952}*{Theorem on p.~174}
shows that $\gamma$ cannot be deformed to a point,
so $U$ is not simply connected.
If $U$ does not contain a simple closed curve,
then the image of any simple closed curve in $U$ is a dendrite, 
and hence by Proposition~\ref{P:contractibility of dendrites} 
is contractible,
so $U$ is simply connected.
\item 
This follows from \eqref{I:simply connected local dendrite}.\qedhere
  \end{enumerate}
\end{proof}

\subsection{Local dendrites and quasi-polyhedra}

We now relate the notion of quasi-polyhedron 
in \cite{Berkovich1990}*{\S4.1}
to the notion of local dendrite.

\begin{proposition}
\label{P:compact metrizable quasi-polyhedron}
\hfill
\begin{enumerate}[\upshape (a)]
\item \label{I:open subset of local dendrite is quasi-polyhedron}
A connected open subset of a local dendrite
is a quasi-polyhedron.
\item \label{I:quasi-polyhedron is local dendrite}
A compact metrizable quasi-polyhedron is the same thing
as a local dendrite.
\item \label{I:quasi-polyhedron is dendrite}
 A compact metrizable simply connected quasi-polyhedron is 
the same thing as a dendrite.
\item A compact metrizable quasi-polyhedron is special
in the sense of \cite{Berkovich1990}*{Definition~4.1.5}.
\end{enumerate}
\end{proposition}

\begin{proof}
\hfill
\begin{enumerate}[\upshape (a)]
\item 
Suppose that $V$ is a connected open subset of a local dendrite~$X$.
By \cite{Kuratowski1968}*{\S51, VII, Theorem~1},
each point $v$ of $V$ has arbitrarily small open neighborhoods
$\UU$ with finite boundary.
We may assume that each $\UU$ is contained in a dendrite.
Since $V$ is locally connected, we may replace each $\UU$
by its connected component containing $x$: this can only
shrink its boundary.
Now each $\UU$, as a connected subset of a dendrite, is 
uniquely arcwise connected~\cite{Whyburn1971}*{p.~89, 1.3(ii)}.
So these $\UU$ satisfy~\cite{Berkovich1990}*{Definition~4.1.1(i)(a)}.

By Proposition~\ref{P:embedding local dendrites}\eqref{I:local dendrite embeds in R^3} (whose proof does not use anything from here on!),
$X$ is homeomorphic to a compact subset of $\R^3$,
so every open subset of $X$ is countable at infinity
(i.e., a countable union of compact sets).
Thus $V$ is a quasi-polyhedron.

\item 
If $X$ is a local dendrite,
it is a quasi-polyhedron 
by~\eqref{I:open subset of local dendrite is quasi-polyhedron}
and compact and metrizable by definition.

Conversely, suppose that $X$ is a compact metrizable quasi-polyhedron.
In particular, $X$ is a continuum.
Condition~($a_2$) in \cite{Berkovich1990}*{Definition~4.1.1}
implies that $X$ is locally connected
and covered by open subsets containing no simple closed curve.
By compactness, this implies that there is a positive lower bound $\epsilon$
on the diameter of simple closed curves in $X$.
By \cite{Kuratowski1968}*{\S51, VII, Lemma~3}, 
this implies that $X$ is a local dendrite.

\item Combine \eqref{I:quasi-polyhedron is local dendrite}
and Proposition~\ref{P:properties of local dendrites}\eqref{I:dendrite vs local dendrite}.

\item 
A dendrite is special since each partial ordering 
as in \cite{Berkovich1990}*{Definition~4.1.5}
arises from some $x \in X$, and we can take $\theta$ there
to be a radial distance function 
as in \cite{Mayer-Oversteegen1990}*{Section~4.6},
which applies since dendrites are locally arcwise connected
and uniquely arcwise connected.
A local dendrite is special since any simply connected sub-quasi-polyhedron
is homeomorphic to a connected open subset of a dendrite.
\qedhere
\end{enumerate}
\end{proof}

\subsection{The core skeleton}

By \cite{Berkovich1990}*{Proposition~4.1.3(i)},
any simply connected quasi-polyhedron $Q$
has a unique compactification $\widehat{Q}$
that is a simply connected quasi-polyhedron.
The points of $\widehat{Q}-Q$ are called the \defi{endpoints} of $Q$.
Given a quasi-polyhedron $X$, 
Berkovich defines its \defi{skeleton} $\Delta(X)$ 
as the complement in $X$ of the set of points
having a simply connected quasi-polyhedral open neighborhood
with a single endpoint~\cite{Berkovich1990}*{p.~76}.
In the case of a local dendrite,
we can characterize this subset in many ways: 
see Proposition~\ref{P:core skeleton}.

\begin{lemma}
\label{L:connected component}
Let $X$ be a local dendrite.
Let $G$ be a subcontinuum of $X$ containing all the simple closed curves.
Let $C$ be a connected component of $X-G$.
Then $C$ is open in $X$
and is a simply connected quasi-polyhedron with one endpoint,
and its closure $\overline{C}$ in $X$ is a dendrite
intersecting $G$ in a single point.
\end{lemma}

\begin{proof}
Since $X$ is locally connected, $X-G$ is locally connected,
so $C$ is open.
By Proposition~\ref{P:compact metrizable quasi-polyhedron}\eqref{I:open subset of local dendrite is quasi-polyhedron},
$C$ is a quasi-polyhedron.
Since $C$ contains no simple closed curve,
it is simply connected by 
Proposition~\ref{P:properties of local dendrites}\eqref{I:simply connected local dendrite}.

The complement of $C \union G$ is a union of connected components of $X-G$,
so $C \union G$ is closed, so it contains $\overline{C}$.
Since $X$ is connected, $\overline{C} \ne C$,
so $\#(\overline{C} \intersect G) \ge 1$.

If $C$ had more than one endpoint, 
there would be an arc $\alpha$ in $\widehat{C}$ connecting two of them,
passing through some $c \in C$ since $\widehat{C}-C$ is totally disconnected 
by \cite{Berkovich1990}*{Proposition~4.1.3(i)};
the image of $\alpha$ under the induced map $\widehat{C} \to X$
together with an arc in $G$ connecting the images of the two endpoints
would contain a simple closed curve passing through $c$,
contradicting the hypothesis on $G$.
Also, each point in $\overline{C} \intersect G$ 
is the image of a point in $\widehat{C}-C$.
Now $1 \le \#(\overline{C} \intersect G) \le \#(\widehat{C}-C) \le 1$,
so equality holds everywhere.
\end{proof}

\begin{proposition}
\label{P:core skeleton}
Let $X$ be a local dendrite.
Each of the following conditions defines the same closed
subset $\Delta$ of $X$.
\begin{enumerate}[\upshape (i)]
\item\label{I:smallest subcontinuum} 
If $X$ is a dendrite, $\Delta=\emptyset$;
otherwise $\Delta$ is the smallest subcontinuum of $X$ containing 
all the simple closed curves.
\item 
\label{I:union of arcs}
The set $\Delta$ is the union of all arcs each endpoint of which
belongs to a simple closed curve.
\item\label{I:Berkovich skeleton} 
The set $\Delta$ is the skeleton $\Delta(X)$ 
defined in \cite{Berkovich1990}*{p.~76}.
\end{enumerate}
\end{proposition}

\begin{proof}
Let $L$ be the union of the simple closed curves in $X$.
If $L=\emptyset$, then $X$ is a dendrite and 
\eqref{I:smallest subcontinuum}, 
\eqref{I:union of arcs}, 
\eqref{I:Berkovich skeleton}
all define the empty set.
So suppose that $L \ne \emptyset$.

For each pair of distinct components of $L$,
there is at most one arc $\alpha$ in $X$ intersecting $L$
in two points, one from each component in the pair
(otherwise there would be a simple closed curve not contained in $L$).
Let $D$ be the union of all these arcs $\alpha$ with $L$.
Any arc $\beta$ in $X$ with endpoints in $L$ must be contained in $D$,
since a point of $\beta$ outside $D$ would be contained in
some subarc $\beta'$ intersecting $L$ in just the endpoints of $\beta'$,
which would then have to be some $\alpha$.
Thus $D$ is the union of the arcs whose endpoints lie in $L$.
By Proposition~\ref{P:properties of local dendrites}\eqref{I:arcwise connected local dendrite}, $X$ is arcwise connected, so $D$ is arcwise connected.
By definition, $D$ is a finite union of compact sets,
so $D$ is a subcontinuum.

By Proposition~\ref{P:properties of local dendrites}\eqref{I:arcwise connected local dendrite},
any subcontinuum $Y \subseteq X$ is arcwise connected,
so if $Y$ contains $L$, then for each $\alpha$ as above, 
$Y$ contains an arc $\beta$ with the same endpoints as $\alpha$,
and then $\beta=\alpha$ 
(otherwise there would be subarcs of $\alpha$ and $\beta$
whose union was a simple closed curve not contained in $L$);
thus $Y \supseteq D$.
Hence $D$ is the smallest subcontinuum containing $L$.

Let $\Delta$ be the $\Delta(X)$ of~\cite{Berkovich1990}*{p.~76}.
If $x$ were a point in a simple closed curve $\gamma$ in $X$ 
with a neighborhood $Q$ as in the definition of $\Delta$,
then $Q$ must contain $\gamma$, since otherwise $Q \intersect \gamma$
would have a connected component 
homeomorphic to an open interval $I$,
and the two points of $\widehat{I}-I$
would map to two distinct points of $\widehat{Q}-Q$,
contradicting the choice of $Q$.
Thus $\Delta \supseteq L$.
But $D$ is the smallest subcontinuum containing $L$,
so $\Delta \supseteq D$.
On the other hand,
Lemma~\ref{L:connected component}
shows that the points of $X-D$ lie outside $\Delta$.
Hence $\Delta=D$.
\end{proof}

We call $\Delta$ the \defi{core skeleton} of $X$,
since in \cite{Hrushovski-Loeser-preprint}*{Section~10}
the term ``skeleton'' is used more generally 
for any finite simplicial complex 
onto which $X$ admits a strong deformation retraction.
If $\Delta \ne \emptyset$,
then $\Delta$ is a finite connected graph with no vertices 
of degree less than or equal 
to $1$~\cite{Berkovich1990}*{Proposition~4.1.4(ii)}.

\subsection{\texorpdfstring{$G$}{G}-dendrites}

\begin{proposition}
\label{P:G-dendrite}
For a subcontinuum $G$ of $X$, the following are equivalent.
\begin{enumerate}[\upshape (i)]
\item\label{I:contains core skeleton} 
$G$ contains the core skeleton of $X$.
\item\label{I:deformation retract}
$G$ is a deformation retract of $X$.
\item\label{I:strong deformation retract}
$G$ is a strong deformation retract of $X$.
\item \label{I:very strong deformation retract}
There is a retraction $r \colon X \to G$
such that 
there exists a homotopy $h \colon [0,1] \times X \to X$
between $h(0,x)=x$ and $h(1,x)=r(x)$
satisfying $r(h(t,x)) = r(x)$ for all $t$ and $x$
(i.e., ``points are moved only along the fibers of $r$'');
moreover, $r$ is unique, 
characterized by the condition that it maps each connected component $C$
of $X-G$ to the singleton $\overline{C} \intersect G$.
\end{enumerate}
\end{proposition}

\begin{proof}
First we show that a retraction $r$ as 
in~\eqref{I:very strong deformation retract}
must be as characterized.
Suppose that $C$ is a connected component of $X-G$.
Any $c \in C$ is moved by the homotopy along a path ending on $G$,
and if we shorten it to a path $\gamma$ 
so that it ends as soon as it reaches $G$
then $\gamma$ stays within $X-G$ until it reaches its final point $g$
and hence stays within $C$ until it reaches $g$;
Hence $g \in \overline{C} \intersect G$, and $r(c)=g$.
Thus $r(C) \subseteq \overline{C} \intersect G$.
By Lemma~\ref{L:connected component}, 
$\#(\overline{C} \intersect G)=1$,
so $r$ is as characterized.

\eqref{I:contains core skeleton}$\Rightarrow$\eqref{I:very strong deformation retract}:
See \cite{Berkovich1990}*{Proposition~4.1.6} and its proof.

\eqref{I:very strong deformation retract}$\Rightarrow$\eqref{I:strong deformation retract}: Trivial.

\eqref{I:strong deformation retract}$\Rightarrow$\eqref{I:deformation retract}: Trivial.

\eqref{I:deformation retract}$\Rightarrow$\eqref{I:contains core skeleton}: 
The result of deforming 
the inclusion of a simple closed curve $\gamma$ in $X$
is a closed path whose 
image contains $\gamma$ \cite{Borsuk-Jaworowski1952}*{Theorem on p.~174},
so if $G$ is a deformation retract of $X$,
then $G$ must contain each simple closed curve, 
so $G$ contains the core skeleton.
\end{proof}

Given an embedding of local dendrites $G \injects X$,
call $X$ equipped with the embedding 
a \defi{$G$-dendrite} if the image of $G$ satisfies the conditions
of Proposition~\ref{P:G-dendrite};
we generally identify $G$ with its image.
Let $\DD_G$ be the category whose objects are $G$-dendrites
and whose morphisms are embeddings extending the identity $1_G \colon G \to G$.
Given a $G$-dendrite $X$ and $g \in G$,
let $X_g$ be the fiber $r^{-1}(g)$ with the point $g$ distinguished;
say that $g$ is a \defi{sprouting point} if $X_g$ is not a point.
Theorem~\ref{T:attaching dendrites} below
makes precise the statement that any $G$-dendrite
is obtained by attaching dendrites to countably many points of $G$.

\begin{theorem}
\label{T:attaching dendrites}
There is a fully faithful functor $F \colon \DD_G \to \prod_{g \in G} \Pointed$
sending a $G$-dendrite $X$ to the tuple of fibers $(X_g)_{g \in G}$, 
and its essential image consists of tuples $(D_g)$
such that $\{g \in G: \# D_g > 1 \}$ is countable.
\end{theorem}

\begin{proof}
Let $X$ be a $G$-dendrite.
For each $g \in G$,
the homotopy restricts to a contraction of $X_g$ to $g$,
so $X_g$ is a (pointed) dendrite.
By \cite{Kuratowski1968}*{\S51, IV, Theorem~5 and \S51, VII, Theorem~1}, 
$\{g \in G: \# X_g > 1 \}$ is countable.

The characterization of the retraction in 
Proposition~\ref{P:G-dendrite}\eqref{I:very strong deformation retract}
shows that a morphism of $G$-dendrites $X \to Y$ respects the retractions,
so it restricts to a morphism $X_g \to Y_g$ in $\Pointed$ for each $g \in G$.
This defines $F$.

Given $(D_g)_{g \in G} \in \prod_{g \in G} \Pointed$
with $\{g \in G: \# D_g > 1 \}$ countable,
choose a metric $d_{D_g}$ on $D_g$ such that the diameters of
the $D_g$ with $\#D_g>1$ tend to $0$ if there are infinitely many of them.
Identify the distinguished point of $D_g$ with $g$.
Let $X$ be the set $\coprod_{g\in G} D_g$
with the metric for which 
the distance between $x \in D_g$ and $x' \in D_{g'}$ is
\[
\begin{cases}
  d_{D_g}(x,x'), & \textup{ if $g=g'$,} \\
  d_{D_g}(x,g) + d_G(g,g') + d_{D_{g'}}(g',x'), & \textup{ if $g \ne g'$.} \\
\end{cases}
\]
It is straightforward to check that $X$ is compact and locally connected
and that the map $G \to X$ is an embedding.
By Proposition~\ref{P:contractibility of dendrites},
there is a strong deformation retraction of $D_g$ onto $\{g\}$;
running these deformations in parallel yields 
a strong deformation retraction of $X$ onto $G$.
Thus $X$ is a $G$-dendrite.
Moreover, $F$ sends $X$ to $(D_g)_{g \in G}$.
Thus the essential image is as claimed.

Given $X,Y \in \DD_G$,
and given morphisms $f_g \colon X_g \to Y_g$ in $\Pointed$ for all $g \in G$,
there exists a unique morphism $f \colon X \to Y$ in $\DD_G$
mapped by $F$ to $(f_g)_{g \in G}$;
namely, one checks that the union $f$ of the $f_g$
is a continuous injection, and hence an embedding.
Thus $F$ is fully faithful.
\end{proof}

\subsection{The universal \texorpdfstring{$G$}{G}-dendrite}

Let $G$ be a local dendrite.
Given a countable subset $G_0 \subseteq G$,
Theorem~\ref{T:attaching dendrites}
yields a $G$-dendrite $W_{G,G_0}$ 
whose fiber at $g \in G$
is the universal pointed dendrite $(W,w)$ if $g \in G_0$
and a point if $g \notin G_0$.
By Theorems \ref{T:attaching dendrites} 
and~\ref{T:universal pointed dendrite}, 
any $G$-dendrite with all sprouting points in $G_0$
admits a morphism to $W_{G,G_0}$.

Now let $G$ be a finite connected graph.
Fix a countable dense subset $G_0 \subseteq G$
containing all vertices of $G$.
Define $W_G \colonequals W_{G,G_0}$,
and call it the \defi{universal $G$-dendrite}. 
Its homeomorphism type is independent of the choice of $G_0$,
since the possibilities for $G_0$ 
are permuted by the self-homeomorphisms of $G$ fixing its vertices.
Any $G$-dendrite has its sprouting points contained in some $G_0$ as above
(just take the union with a $G_0$ from above),
so every $G$-dendrite embeds as a topological space into $W_G$.

\begin{theorem}
\label{T:local dendrite is W_G}
Let $X$ be a local dendrite, and let $G$ be its core skeleton.
Suppose that $G \ne \emptyset$,
that the branch points of $X$ are dense in $X$,
and that there are $\aleph_0$ branches at each branch point.
Then $X$ is homeomorphic to $W_G$.
\end{theorem}

\begin{proof}
The vertices of $G$ of degree~$3$ or more are among the branch points of $X$.
After applying a homeomorphism of $G$ (to shift degree~$2$ vertices),
we may assume that \emph{all} the vertices of $G$ are branch points of $X$.
Since the branch points of $X$ are dense in $X$,
the sprouting points must be dense in $G$.
For each sprouting point $g \in G$,
the fiber $X_g$ satisfies the hypotheses of 
Theorem~\ref{T:Wazewski}, so $X_g$ is the universal pointed dendrite.
Thus $X$ is homeomorphic to $W_G$, by construction of the latter.
\end{proof}

\subsection{Euclidean embeddings}

\begin{proposition}\hfill
\label{P:embedding local dendrites}
\begin{enumerate}[\upshape (a)]
\item\label{I:local dendrite embeds in R^3} 
Every local dendrite embeds in $\R^3$.
\item Let $X$ be a local dendrite, 
and let $G \subseteq X$ be a finite connected graph
containing all the simple closed curves.
Then the following are equivalent:
\begin{enumerate}[\upshape (i)]
\item $X$ embeds into $\R^2$.
\item $G$ embeds into $\R^2$.
\item $G$ does not contain a subgraph isomorphic to 
a subdivision of the complete graph $K_5$ 
or the complete bipartite graph $K_{3,3}$.
\end{enumerate}
\end{enumerate}
\end{proposition}

\begin{proof}\hfill
  \begin{enumerate}[\upshape (a)]
  \item 
A local dendrite is 
a regular continuum \cite{Kuratowski1968}*{\S51, VII, Theorem~1},
and hence of dimension~$1$,
so it embeds in $\R^3$ by as discussed in Remark~\ref{R:history}.
  \item See~\cite{Kuratowski1930}.\qedhere
  \end{enumerate}
\end{proof}

\section{Berkovich curves}\label{S:Berkovich curves}

Finally, we build on \cite{Berkovich1990} (especially Section~4 therein)
and the theory of local dendrites to describe the homeomorphism type
of a Berkovich curve. 
See also the forthcoming 
book by A.~Ducros~\cite{Ducros-preprint},
which will contain a systematic study of Berkovich curves.

\begin{theorem}
\label{T:Berkovich space and local dendrites}
Let $K$ be a complete valued field having a countable dense subset.
Let $V$ be a projective $K$-scheme of pure dimension~$1$.
\begin{enumerate}[\upshape (a)]
\item\label{I:union of local dendrites}
The topological space $V^{\an}$ 
is a finite disjoint union of local dendrites.
\item\label{I:Berkovich space is universal dendrite}
Suppose that $V$ is also smooth and connected, 
and that $K$ has nontrivial value group.
\begin{enumerate}[\upshape (i)]
\item If $V^{\an}$ is simply connected,
then $V^{\an}$ is homeomorphic to the Wa{\.z}ewski universal dendrite $W$.
\item If $V^{\an}$ is not simply connected, let $G$ be its core skeleton;
then $V^{\an}$ is homeomorphic to the universal $G$-dendrite $W_G$.
\end{enumerate}
\end{enumerate}
\end{theorem}

\begin{proof}
\hfill
\begin{enumerate}[\upshape (a)]
\item 
We may assume that $V$ is connected,
so $V^{\an}$ is connected by~\cite{Berkovich1990}*{Theorem~3.4.8(iii)}.
Also, $V^{\an}$ is compact by~\cite{Berkovich1990}*{Theorem~3.4.8(ii)}.
It is metrizable by Theorem~\ref{T:main}.
It is a quasi-polyhedron by 
the proof of~\cite{Berkovich1990}*{Corollary~4.3.3}.
So $V^{\an}$ is a local dendrite 
by Proposition~\ref{P:compact metrizable quasi-polyhedron}.
\item 
Let $k$ be the residue field of $K$.
Since $K$ has a countable dense subset, $k$ is countable,
so any $k$-curve has exactly $\aleph_0$ closed points.

First suppose that $K$ is algebraically closed.  In particular $K$ has dense value group.
Choose a semistable decomposition of $V^{\an}$ 
(see \cite{Baker-Payne-Rabinoff-preprint}*{Definition~5.15}).
Each open ball and open annulus in the decomposition
is homeomorphic to an open subspace of $(\PP^1_K)^{\an}$,
in which the branch points 
(type~(2) points in the terminology of \cite{Berkovich1990}*{1.4.4}) 
are dense by the assumption on the value group, 
so the branch points are dense in $V^{\an}$.
At each branch point,
the branches are in bijection with the closed points of a $k$-curve
by \cite{Baker-Payne-Rabinoff-preprint}*{Lemma~5.66(3)},
so their number is $\aleph_0$.

Now suppose that $K$ is not necessarily algebraically closed.
Let $K'$ be the completion of an algebraic closure of $K$.
Then \cite{Berkovich1990}*{Corollary~1.3.6} implies that
$V^{\an}$ is the quotient of $(V_{K'})^{\an}$ by
the absolute Galois group of $K$.
It follows that the branch points of $V^{\an}$
are the images of the branch points of $(V_{K'})^{\an}$,
and that the branches at each branch point of $V^{\an}$
are in bijection with the closed points of some 
curve over a finite extension of $k$.
Thus, as for $(V_{K'})^{\an}$, 
the branch points of $V^{\an}$ are dense,
and there are $\aleph_0$ branches at each branch point.

Finally, according to whether $G$ is simply connected or not,
Theorem~\ref{T:Wazewski} or Theorem~\ref{T:local dendrite is W_G}
shows that $V^{\an}$ has the stated homeomorphism type.
\qedhere
\end{enumerate}
\end{proof}

\begin{corollary}
\label{C:P^1}
Let $K$ be a complete valued field having a countable dense subset
and dense value group.
Then $(\PP^1_K)^{\an}$ is homeomorphic to $W$.
\end{corollary}

\begin{proof}
It is simply connected by~\cite{Berkovich1990}*{Theorem~4.2.1}, 
so Theorem~\ref{T:Berkovich space and local dendrites}\eqref{I:Berkovich space is universal dendrite}(i) 
applies.
\end{proof}

\begin{remark}
Any finite connected graph with no vertices of degree 
less than or equal to~$1$ can arise
as the core skeleton $G$
in Theorem~\ref{T:Berkovich space and local dendrites}\eqref{I:Berkovich space is universal dendrite}(ii):
see~\cite{Berkovich1990}*{proof of Corollary~4.3.4}.
In particular, there exist smooth projective curves $V$ such that 
$V^{\an}$ cannot be embedded in $\R^2$.
\end{remark}

\begin{remark}
Theorem~\ref{T:Berkovich space and local dendrites}
also lets us understand the topology of Berkovich spaces
associated to curves that are only \emph{quasi-projective}.
Let $U$ be a quasi-projective curve.
Write $U = V-Z$ for some projective curve $V$
and finite subscheme $Z \subseteq V$.
Then $Z^{\an}$ is a closed subset of $V^{\an}$ 
with one point for each closed point of $Z$,
and $U^{\an} = V^{\an} - Z^{\an}$.
\end{remark}

\begin{remark}
Even more generally, the arguments apply equally well to
Berkovich curves that do not arise as analytification of algebraic curves.
\end{remark}

\begin{remark}
The smoothness assumption in 
Theorem~\ref{T:Berkovich space and local dendrites}\eqref{I:Berkovich space is universal dendrite} 
can be weakened to the statement that 
the normalization morphism $\widetilde{V} \to V$ 
has no fibers with three or more schematic points.     
\end{remark}

\begin{remark}  
If in Theorem~\ref{T:Berkovich space and local dendrites}\eqref{I:Berkovich space is universal dendrite} 
we drop any of the hypotheses, 
then the result fails; we describe the situations that arise.
\begin{itemize}

\item 
If $V$ is the non-smooth curve consisting of three copies of $\PP^1_K$
attached at a $K$-point of each,
then $V^{\an}$ consists of three copies of $W$ attached in the same way;
this is a dendrite, but it has a branch point of order~$3$,
so it cannot be homeomorphic to $W$.  
More generally, if the normalization $\widetilde{V}$ 
has three distinct schematic points above some point $a$ of $V$,
the same argument applies.   
\item 
If $V$ is disconnected, then so is $V^{\an}$, so it cannot be 
homeomorphic to $W$ or $W_G$. 
In this case, $V^{\an}$ is the disjoint union 
of the analytifications of the connected components of $V$.  
\item 
Suppose that $V$ is smooth and connected, but $K$ has trivial value group.
Then $V^{\an}$ is a dendrite consisting of $\aleph_0$ intervals
emanating from one branch point; cf.~\cite{Berkovich1993}*{p.~71}.
Equivalently, $V^{\an}$ is the one-point compactification
of $|V| \times [0,\infty)$,
where $|V|$ is the set of closed points of $V$ with the discrete topology.
  \end{itemize}
\end{remark}

\begin{remark}
As is well-known to 
experts~\cites{Thuillier-thesis, Baker-Payne-Rabinoff-preprint},
there is a metrized variant of 
Theorem~\ref{T:Berkovich space and local dendrites}.
We recall a few definitions; cf.~\cite{Mayer-Nikiel-Oversteegen1992}.
An \defi{$\R$-tree} is a uniquely arcwise connected metric space
in which each arc is isometric to a subarc of $\R$.
Let $A$ be a countable subgroup of $\R$,
and let $A_{\ge 0}$ (resp.\ $A_{>0}$) 
be the set of nonnegative (resp.\ positive) numbers in $A$.
An \defi{$A$-tree} is an $\R$-tree $X$ equipped with a point $x \in X$
such that the distance from each branch point to $x$ lies in $A$.

More generally, we may introduce variants that are not simply connected.
Let us define an \defi{$\R$-graph} to be an arcwise connected metric space
$X$ such that each arc of $X$ is isometric to a subarc of $\R$
and $X$ contains at most finitely many simple closed curves.
Define an \defi{$A$-graph} to be an $\R$-graph $X$
equipped with a point $x \in X$
such that the length of every arc from $x$ to a branch point or to itself 
is in $A$.
Given an $A$-graph $(X,x)$,
let $B(X)$ be the set of points $y \in B$ not of degree~$1$
such that $y$ is an endpoint of an arc of length in $A_{\ge 0}$ 
emanating from $x$.
Then let $\calE(X)$ be the $A$-graph obtained
by attaching $\aleph_0$ isometric copies of $[0,\infty)$
and of $[0,a]$ for each $a \in A_{>0}$
to each $y \in B(X)$ (i.e., identify each $0$ with $y$).
Let $\calE^n(X)\colonequals \calE(\calE(\cdots(\calE(X))\cdots))$.
The direct limit of the $\calE^n(X)$ is an $A$-graph $\calW^A_X$.
If $X$ is a point, define $\calW^A \colonequals \calW^A_X$,
which is a universal separable $A$-tree in the sense
of \cite{Mayer-Nikiel-Oversteegen1992}*{Section~2},
because it contains the space obtained by attaching
only copies of $[0,\infty)$ at each stage;
the latter is the universal separable $A$-tree 
constructed in \cite{Mayer-Nikiel-Oversteegen1992}*{Theorem~2.6.1}.

Let $K$ be a complete valued field having a countable dense subset.
Let $A_0$ be the value group of $K$, expressed as an additive subgroup of $\R$.
Let $A \le R$ be the $\Q$-vector space spanned by $A_0$.
Let $V$ be a projective $K$-scheme of pure dimension~$1$.
Let $V^{\anm}$ be the subset of $V^{\an}$ 
consisting of the complement of the type~(1) points
(the points corresponding to closed points of $V$).
Then $V^{\anm}$ admits a canonical metric,
whose existence is related to the fact that on the segments of 
the skeleta of $V^{\an}$, away from the endpoints, 
one has an integral affine 
structure \cite{Kontsevich-Soibelman2006}*{Section~2}.
If $V^{\anm}$ is simply connected,
then $V^{\anm}$ is isometric to $\calW^A$;
otherwise $V^{\anm}$ is isometric to $\calW^A_G$,
where $G$ is the core skeleton of $V^{\an}$ with the induced metric.
\end{remark}

\begin{warning}
The metric topology on $V^{\anm}$ is strictly stronger than
the subspace topology on $V^{\anm}$ induced from $V^{\an}$: 
see \cite{Favre-Jonsson2004}*{Chapter~5} 
and \cite{Baker-Rumely2010}*{Section~B.6}.
Nevertheless, when $V$ is smooth and complete, 
the topological space $V^{\an}$ can be recovered 
from the metric space $V^{\anm}$.
\end{warning}

\section*{Acknowledgements} 

We thank Sam Payne for suggesting the argument and references 
for the facts that a Berkovich curve satisfying our hypotheses
has a dense set of branch points and that each has $\aleph_0$ branches.
The third author thanks the Centre Interfacultaire Bernoulli for its
hospitality.

\end{document}